\pdfoutput=1
\documentclass[11pt,a4paper]{amsart}
\usepackage[margin=1in]{geometry}
\usepackage[utf8]{inputenc}
\usepackage[T1]{fontenc}
\usepackage[dvipsnames]{xcolor}
\usepackage{graphicx}
\usepackage{microtype}
\usepackage{fnpct}

\usepackage{amsmath}
\usepackage{amssymb}
\usepackage[mathscr]{eucal}
\usepackage{xfrac}
\usepackage{mathtools}
\usepackage{tikz-cd}
\renewcommand{\injlim}{\varinjlim}
\renewcommand{\projlim}{\varprojlim}

\usepackage{enumitem}
\usepackage{booktabs}
\usepackage[pdfusetitle,colorlinks]{hyperref}
\hypersetup{bookmarksdepth=2,pdfencoding=unicode,allcolors=MidnightBlue}
\usepackage[capitalise,noabbrev,nosort]{cleveref}

\crefname{equation}{}{}
\crefformat{enumi}{(#2#1#3)}
\crefmultiformat{enumi}{(#2#1#3)}{ and~(#2#1#3)}{, (#2#1#3)}{, and~(#2#1#3)}
\crefname{enumi}{}{}
\newlist{conenum}{enumerate}{1}
\setlist[conenum,1]{label=(\roman*),ref=\roman*}
\creflabelformat{conenumi}{(#2#1#3)}
\crefname{conenumi}{}{}

\numberwithin{equation}{section}
\theoremstyle{plain}
\newtheorem{Theorem}{Theorem}

\crefname{Theorem}{Theorem}{Theorems}
\newtheorem{theorem}[equation]{Theorem}
\newtheorem{proposition}[equation]{Proposition}
\newtheorem{lemma}[equation]{Lemma}
\newtheorem{corollary}[equation]{Corollary}

\theoremstyle{definition}
\newtheorem{definition}[equation]{Definition}
\newtheorem{example}[equation]{Example}
\newtheorem{question}[equation]{Question}

\theoremstyle{remark}
\newtheorem{remark}[equation]{Remark}

\let\oldSS\SS\let\SS\relax

\newcommand{\CC}{\mathbb{C}}
\newcommand{\SS}{\mathbb{S}}

\newcommand{\E}{\mathrm{E}}

\newcommand{\h}{\textnormal{h}}
\newcommand{\fin}{\textnormal{fin}}
\newcommand{\idem}{\textnormal{idem}}
\newcommand{\sm}{\textnormal{sm}}

\newcommand{\Alg}{\operatorname{Alg}}
\newcommand{\CAlg}{\operatorname{CAlg}}
\newcommand{\Down}{\operatorname{Down}}
\newcommand{\D}{\operatorname{D}}
\newcommand{\End}{\operatorname{End}}
\newcommand{\Fun}{\operatorname{Fun}}
\newcommand{\Idem}{\operatorname{Idem}}
\newcommand{\Idl}{\operatorname{Idl}}
\newcommand{\Ind}{\operatorname{Ind}}
\newcommand{\KU}{\operatorname{KU}}
\newcommand{\Loc}{\operatorname{Loc}}
\newcommand{\Map}{\operatorname{Map}}

\newcommand{\PShv}{\operatorname{PShv}}

\newcommand{\Shv}{\operatorname{Shv}}
\newcommand{\Sm}{\operatorname{Sm}}
\newcommand{\Sub}{\textnormal{Sub}}
\newcommand{\st}{\textnormal{st}}
\newcommand{\Spec}{\operatorname{Spec}}
\newcommand{\cCAlg}{\operatorname{cCAlg}}
\newcommand{\cIdem}{\operatorname{cIdem}}
\newcommand{\cLoc}{\operatorname{cLoc}}
\newcommand{\cofib}{\operatorname{cofib}}
\newcommand{\fib}{\operatorname{fib}}
\newcommand{\id}{\operatorname{id}}
\newcommand{\op}{\operatorname{op}}

\newcommand{\X}{\mathord{-}}
\newcommand{\unit}{\mathbf{1}}

\newcommand{\cat}[1]{\mathscr{#1}}
\newcommand{\cls}[1]{\mathcal{#1}}
\newcommand{\Cat}[1]{\mathsf{#1}}
\newcommand{\shf}[1]{\mathcal{#1}}

\title{The sheaves--spectrum adjunction}
\author{Ko Aoki}
\address{Max Planck Institute for Mathematics,
  Vivatsgasse 7, 53111 Bonn, Germany
}
\email{aoki@mpim-bonn.mpg.de}
\date{\today}

\begin{document}

\begin{abstract}
  This paper demystifies the notion of the smashing spectrum
  of a stable presentably symmetric monoidal \(\infty\)-category,
  defined as a locale whose opens
  correspond to smashing localizations.
  Previously, this concept was studied
  in tensor-triangular geometry in
  the compactly generated rigid setting.
  Our main result identifies the smashing spectrum functor
  as the right adjoint to the spectral sheaves functor,
  providing in particular an external characterization
  that avoids explicit reference to objects, ideals, or localizations.
  The sheaves–spectrum adjunction formalizes the intuition
  that the smashing spectrum constitutes the best approximation
  of a given \(\infty\)-category by \(\infty\)-categories of sheaves.
  We establish an unstable generalization of this result
  by identifying the correct unstable analog of the
  smashing spectrum,
  which parametrizes smashing colocalizations instead.

  As an application,
  we give a categorical presentation
  of Clausen–Scholze’s categorified locales.
\end{abstract}

\maketitle
\setcounter{tocdepth}{1}
\tableofcontents

\section{Introduction}\label{s:intro}

\subsection{Smashing spectrum and categorified locales}\label{ss:q}

For a tensor triangulated category~\(\cat{D}\),
the notion of an \emph{idempotent algebra},
i.e., a (commutative) algebra object~\(E\)
whose multiplication morphism \(E\otimes E\to E\) is an isomorphism,
is well known,
as seen in~\cite[Exercise~4.2]{KashiwaraSchapira06}.
It was systematically studied
by Boyarchenko–Drinfeld~\cite{BoyarchenkoDrinfeld14}
and
Balmer–Favi~\cite{BalmerFavi11},
who observed that
such an algebra corresponds to a \emph{smashing localization},
i.e., a localization~\(L\colon\cat{D}\to\cat{D}\)
that is equivalent to \(L(\unit)\otimes\X\).
Balmer–Favi studied the poset of idempotent algebras
within \emph{tensor-triangular geometry},
a subject that studies tensor triangulated categories
by associating spaces to them.
Subsequently,
under the assumption
that \(\cat{D}\) is compactly generated and rigid,
Balmer–Krause–Stevenson~\cite{BalmerKrauseStevenson20}
proved that the poset of idempotent algebras is a \emph{frame};
i.e., it shares
the lattice-theoretic properties
of the poset of open subsets of a topological space.
The \emph{smashing spectrum} is the corresponding “topological space”,
which is formally called a \emph{locale}.

Meanwhile,
Clausen and Scholze launched \emph{condensed mathematics}
(cf.~\cite{Condensed}),
a framework for handling
algebraic structures
in the presence of topology.
Their aim was to recast analytic geometry
in algebraic terms.
They delivered a course (see~\cite{Complex})
and proved standard theorems in complex-analytic geometry
without invoking such “low-level” notions as smooth functions.
Their major tool is the notion of a liquid vector space (cf.~\cite{Analytic}),
but
their framework is also unique in that their model of structured spaces
is no longer locally ringed spaces:
They first noted that
for a cocomplete closed stable symmetric monoidal \(\infty\)-category,
the poset of idempotent algebras is readily seen to be a frame.
Hence we have the smashing spectrum\footnote{It is not necessarily spatial
  in this generality,
  as proven in~\cite{ttg-shv}.
  Hence this use of the term “spectrum”
  does not align with the convention taken in
  tensor triangulated geometry.
} functor
\begin{equation}
  \label{e:sm}
  \phantom,
  {\Sm}
  \colon
  \{\text{cocomplete closed stable symmetric monoidal \(\infty\)-categories}\}
  \to
  {\widehat{\Cat{Loc}}}^{\op}
  ,
\end{equation}
where \(\widehat{\Cat{Loc}}\) denotes the category of large locales,
i.e., the opposite of the category of large frames.
Then
they introduced the notion of a \emph{categorified locale},
which is a triple consisting of
a locale~\(X\),
a cocomplete closed stable symmetric monoidal \(\infty\)-category~\(\cat{C}\),
and
a morphism \(\Sm(\cat{C})\to X\)
of large locales.
Their (suitably generalized) complex-analytic spaces
are defined using the language of categorified locales
(see~\cite[Definition~11.14]{Complex}).
This machinery has proven highly effective;
e.g.,
they were able to prove
the GAGA theorem
for proper morphisms without using Chow’s lemma
(see~\cite[Proposition~13.8]{Complex}).

\begin{remark}\label{71087630bd}
  For their condensed approach to analytic geometry,
  one must allow
  stable presentably symmetric monoidal \(\infty\)-categories that are not rigid;
  e.g.,
  the derived \(\infty\)-category
  of liquid \(\CC\)-vector spaces \(\Cat{Liq}_{\CC}\),
  i.e., liquid quasicoherent modules of
  the one-point complex-analytic space,
  is not rigid.
\end{remark}

Having traced the origins and advantages of the smashing spectrum construction,
we now face the following question:

\begin{question}\label{2d14480999}
  Is there a categorical origin of the smashing spectrum functor~\(\Sm\)
  or of the concept of a categorified locale?
\end{question}

This paper aims to answer this question.

\subsection{Results}\label{ss:main}

We prove the following,
which gives a concise characterization of~\(\Sm\):

\begin{Theorem}\label{main_sp}
  The spectral sheaves functor
  \begin{equation}
    \label{e:shv_sp}
    (\Shv(\X;\Cat{Sp}),{\otimes})\colon
    \Cat{Loc}^{\op}\to\CAlg(\Cat{Pr}_{\st})
  \end{equation}
  from the opposite of the category of locales
  to the \(\infty\)-category of
  stable presentably symmetric monoidal \(\infty\)-categories
  has a right adjoint, which is given by restricting
  the smashing spectrum functor~\cref{e:sm}.
\end{Theorem}

\begin{remark}\label{8875b12b73}
  In \cref{main_sp}, we impose the presentability condition
  in order to ensure that
  \(\Sm\) takes small values.
  This is not a real limitation
  when using categorified locales in analytic geometry,
  since one may work within
  \(\kappa\)-condensed mathematics
  for a fixed cutoff~\(\kappa\).
\end{remark}

We prove a more fundamental variant of this theorem.
First,
note that \cref{e:shv_sp} is the composite
of the (space-valued) sheaves functor
\begin{equation}
  \label{e:shv}
  ({\Shv(\X)},{\times})
  \colon\Cat{Loc}^{\op}
  \to
  \CAlg(\Cat{Pr})
\end{equation}
and the stabilization functor \(\Cat{Sp}\otimes\X\).
Our main result in this paper is
the following:

\begin{Theorem}\label{main}
  The sheaves functor \cref{e:shv} has a right adjoint.
  It is pointwise given by assigning
  to a presentably symmetric monoidal \(\infty\)-category
  the locale
  whose opens are
  (equivalence classes of) idempotent coalgebras,
  or equivalently, smashing colocalizations.
\end{Theorem}

\begin{remark}\label{25cca44437}
  The lower semilattice of
  idempotent coalgebras
  of a general symmetric monoidal \(\infty\)-category
  has a similar characterization;
  see \cref{adj}.
  We use this fact in the proof of \cref{main}.
\end{remark}

\begin{definition}\label{1291cbf8a8}
  We call the right adjoint in \cref{main}
  the \emph{unstable smashing spectrum} functor
  and denote it by the same symbol~\(\Sm\).
\end{definition}

Conversely,
\cref{main} yields
an alternative definition of \(\Shv\).

The relation between \cref{main_sp,main}
is diagrammatically summarized as
\begin{equation*}
  \begin{tikzcd}[column sep=large]
    \Cat{Loc}^{\op}\ar[r,shift left,"\Shv"]
    \ar[rr,bend left=25,"\Shv(\X;\Cat{Sp})"]&
    \CAlg(\Cat{Pr})\ar[r,shift left,"\Cat{Sp}\otimes\X"]\ar[l,shift left,"\Sm"]&
    \CAlg(\Cat{Pr}_{\st})\rlap.\ar[l,shift left,"\textnormal{forget}"]
    \ar[ll,bend left=25,"\Sm"']
  \end{tikzcd}
\end{equation*}
The key difference from the stable case
is that we must work with idempotent coalgebras, not algebras.

\begin{remark}\label{663d806c7e}
  There is a known attempt to obtain a locale
  (or locale-like structure)
  from a symmetric monoidal category of unstable nature.
  For example,
  Enrique~Moliner–Heunen–Tull~\cite{EnriqueMolinerHeunenTull20}
  considered coidempotent objects
  whose structure morphism is a monomorphism.
  However, in the stable case,
  this approach collapses,
  since
  every monomorphisms in
  a stable \(\infty\)-category
  is an equivalence.
\end{remark}

An immediate corollary of
\cref{main}
is that
the functor
\({\Shv}\colon\Cat{Loc}^{\op}\to\CAlg(\Cat{Pr})\)
preserves colimits.
For example, we have the following
categorical Künneth formula:

\begin{corollary}\label{760c2c8226}
  For a diagram \(X'\to X\gets Y\) of locales,
  there is a canonical equivalence
  \begin{equation*}
    \Shv(X'\times_{X}Y)
    \simeq
    \Shv(X')\otimes_{\Shv(X)}\Shv(Y)
  \end{equation*}
  in \(\CAlg(\Cat{Pr})\).
\end{corollary}

\begin{remark}\label{8c3de052e9}
Note that an \(\infty\)- (or even \(1\)-)toposic analog
  of \cref{760c2c8226} dose not hold;
  i.e.,
  a pullback of \(\infty\)-toposes
  cannot be computed as
  a pushout in \(\CAlg(\Cat{Pr})\) in general.
  This may be viewed as an obstruction
  for extending our sheaves–spectrum adjunction
  to the \(\infty\)- (or even \(1\)-)toposic situation.
\end{remark}

Returning to the topic of categorified locales,
\cref{main} says that
a map \(\Sm(\cat{C})\to X\),
which is part of the datum defining a categorified locale \((X,\cat{C})\),
just corresponds to a symmetric monoidal colimit-preserving functor
\(\Shv(X)\to\cat{C}\).
By combining the fully faithfulness of \cref{e:shv},
we can deduce the following
Tannaka duality for categorified locales:

\begin{Theorem}\label{pt}
  Let \(\Cat{CatLoc}_{\st}\)
  denote the \(\infty\)-category of
  stable presentable categorified locales
  (see \cref{catloc} for a precise definition).
  There is a fully faithful functor
  \begin{equation*}
    (\Cat{CatLoc}_{\st})^{\op}\to\Fun(\Delta^1,\CAlg(\Cat{Pr}))
  \end{equation*}
  such that
  an object
  \(\cat{X}\to\cat{C}\)
  of \(\Fun(\Delta^1,\CAlg(\Cat{Pr}))\)
  is in its essential image
  if and only if
  \(\cat{X}\) is equivalent to \(\Shv(X)\)
  for a (unique) locale~\(X\)
  and \(\cat{C}\) is stable.
\end{Theorem}

Finally,
we point out another application of our adjunction.
The smashing spectrum of
\(\Shv(\cat{X};\D(k))\) for an arbitrary \(\infty\)-topos~\(\cat{X}\)
and a field~\(k\)
is determined in~\cite[Theorem~A]{ttg-shv}.
A special case is the following:

\begin{theorem}[\cite{ttg-shv}]\label{sm_shv_k}
  For a locale~\(X\) and a field~\(k\),
  the canonical map \(\Sm(\Shv(X;\D(k)))\to X\)
  of locales is an isomorphism.
\end{theorem}

By combining this with \cref{main} (or \cref{main_sp}),
we immediately obtain the following,
which may be viewed as a Tannaka-type reconstruction result for locales:

\begin{corollary}\label{ee20bfcf81}
  The functor
  \(\Shv(\X;\D(k))\colon\Cat{Loc}^{\op}\to\CAlg_{\D(k)}(\Cat{Pr})\)
  is fully faithful.
\end{corollary}

\begin{remark}\label{8e5ba20a00}
  In the statement of \cref{ee20bfcf81},
  both \(\infty\)-categories have canonical bicategorical structures.
  We can show that
  the functor is fully faithful in the bicategorical sense;
  see~\cite[Theorem~3.8]{ttg-shv}.
\end{remark}

\begin{remark}\label{3e273c7114}
  \Cref{main} (or \cref{main_sp})
  is also useful to see the failure of
  a certain Tannaka-type reconstruction statement.
  For example,
  the nontriviality of \(\Sm(\Cat{Sp})\)
  shows that \cref{e:shv_sp} fails to be fully faithful.
  We can see this more concretely:
  Let \(S\) denote the Sierpiński space,
  i.e., the spectrum of a discrete valuation ring.
  Then \(\Shv(S;\Cat{Sp})\)
  can be identified with \(\Fun(\Delta^1,\Cat{Sp})\)
  with the pointwise symmetric monoidal structure.
  Then \(\Fun(\Delta^1,\Cat{Sp})\to\Cat{Sp}\)
  given by
  \begin{equation*}
    (F(0)\to F(1))
    \mapsto
    \fib\bigl(
      F(1)\to L_{\KU}\cofib(F(0)\to F(1))
    \bigr)
  \end{equation*}
  refines to a symmetric monoidal functor,
  but it does not come from any locale map
  \({*}\to S\).

  However,
  when restricting to the full subcategory
  of compact Hausdorff spaces (aka compact regular locales),
  \(\Shv(\X;\Cat{Sp})\) is fully faithful
  as a functor valued in \(\CAlg(\Cat{Pr}_{\st})\),
  as proven in~\cite{ttg-con}.
  The proof uses the “continuous” version of the sheaves–spectrum adjunction,
  which is the subject of~\cite{ttg-con}.
\end{remark}

\subsection{Organization of this paper}\label{ss:outline}

In \cref{s:idem},
we review
idempotent algebras
and smashing localizations and their duals.
In \cref{s:usm},
we introduce the unstable smashing spectrum functor
and prove \cref{main}.
In \cref{s:catloc},
we review the notion of categorified locales
and
prove \cref{pt}.

\subsection*{Acknowledgments}\label{ss:ack}

I thank
Peter Scholze
for useful discussions
and encouragement.
I thank
Peter Scholze
and
Marco Volpe
for comments on a draft.
I thank
the Max Planck Institute for Mathematics for its financial support.

\subsection*{Conventions}\label{ss:conv}

Following~\cite{LurieHTT},
we use the mark \(\widehat{\X}\) to
denote the \(\infty\)-category of large objects;
e.g., \(\widehat{\Cat{S}}\) and \(\widehat{\Cat{Cat}}\)
denote the very large \(\infty\)-category of large spaces
and large \(\infty\)-categories, respectively.

We write \(\Cat{Pr}\) for the \(\infty\)-category
of presentable \(\infty\)-categories
and colimit-preserving functors.
It is denoted
by \(\Cat{Pr}^{\textnormal{L}}\)
in~\cite{LurieHTT,LurieHA}.
We use the closed symmetric monoidal structure
on~\(\Cat{Pr}\) given in~\cite[Section~4.8.1]{LurieHA}.
We write \(\Cat{Pr}_{\st}\) for its full subcategory spanned by
stable presentable \(\infty\)-categories.

In a poset,
\(\bot\) denotes the least element.

\section{(Co)idempotent algebras and smashing (co)localizations}\label{s:idem}

We review the basic definitions and facts of smashing localizations.
In \cref{ss:e},
we introduce the posets
of idempotent objects,
idempotent algebras,
and
smashing localizations.
We then show that they are mutually isomorphic:
\begin{equation*}
  \Loc_{\sm}(\cat{C})
  \simeq
  \CAlg_{\idem}(\cat{C})
  \simeq
  \Idem(\cat{C}).
\end{equation*}
Then we introduce their dual concepts in \cref{ss:c}:
\begin{equation*}
  \cLoc_{\sm}(\cat{C})
  \simeq
  \cCAlg_{\idem}(\cat{C})
  \simeq
  \cIdem(\cat{C}).
\end{equation*}
We see in \cref{ss:ce}
that in the stable case,
these six posets are isomorphic.
We also see that there
is yet another presentation of this poset via ideals.
However, note again
that in the (presentable) unstable situation
what matters are coidempotent objects.

\subsection{Idempotent objects}\label{ss:e}

We recall the following
from~\cite[Definitions~4.8.2.1 and~4.8.2.8]{LurieHA}:

\begin{definition}\label{3362524e4e}
  Let \(\cat{C}\) be a symmetric monoidal \(\infty\)-category.
  \begin{enumerate}
    \item
      An \emph{idempotent object} is
      an object \(e\colon\unit\to E\)
      of the undercategory~\(\cat{C}_{\unit/}\)
      such that \({\id}_E\otimes e\colon E\otimes\unit
      \to E\otimes E\)
      is an equivalence.
      We write \(\Idem(\cat{C})\) for
      the full subcategory of \(\cat{C}_{\unit/}\)
      spanned by idempotent objects.
    \item
      An \emph{idempotent algebra} is
      a commutative algebra object~\(E\)
      whose binary multiplication \(E\otimes E\to E\)
      is an equivalence.
      We write \(\CAlg_{\idem}(\cat{C})\)
      for the full subcategory of \(\CAlg(\cat{C})\)
      spanned by idempotent algebras.
  \end{enumerate}
\end{definition}

\begin{remark}\label{76b74ce1fe}
  An idempotent object is called 
  a “right idempotent” in~\cite{BalmerFavi11}.
\end{remark}

As proven in~\cite[Proposition~4.8.2.9]{LurieHA},
these two concepts are equivalent:

\begin{proposition}[Lurie]\label{b68cfeb026}
  For a symmetric monoidal \(\infty\)-category~\(\cat{C}\),
  the forgetful functor
  \begin{equation*}
    \CAlg(\cat{C})
    \simeq
    \CAlg(\cat{C})_{\unit/}
    \to
    \cat{C}_{\unit/}
  \end{equation*}
  restricts to an equivalence
  \(\CAlg_{\idem}(\cat{C})\to\Idem(\cat{C})\).
\end{proposition}

We first recall the notion of smashing localization.
See~\cite[Section~5.2.7]{LurieHTT} for general facts
about localizations in the \(\infty\)-categorical setting.

\begin{definition}\label{03a6f7ab93}
  Let \(\Loc(\cat{C})\) denote
  the full subcategory of \(\End(\cat{C})_{\id/}\)
  spanned by localizations.
  We say that a localization~\(L\) is \emph{smashing}
  if \(L\) is equivalent to \(L(\unit)\otimes\X\).
  We let \(\Loc_{\sm}(\cat{C})\) denote
  the full subcategory
  of \(\Loc(\cat{C})\)
  spanned by smashing localizations.
\end{definition}

The following is well known;
e.g.,it was observed after~\cite[Definition~3-2]{GepnerGrothNikolaus15}
and implicit in~\cite[Section~4.8.2]{LurieHA}:

\begin{proposition}\label{e238c29d68}
  For a symmetric monoidal \(\infty\)-category~\(\cat{C}\),
  the functor
  \begin{align*}
    \Loc(\cat{C})
    &\to
    \cat{C}_{\unit/};&
    ({\id}\to L)\mapsto(\unit\to L\unit)
  \end{align*}
  restricts to an equivalence
  \(\Loc_{\sm}(\cat{C})\to\Idem(\cat{C})\).
\end{proposition}

\begin{remark}\label{efeaa9c59c}
  In the triangulated case,
  an analog of \cref{e238c29d68}
  is proven in~\cite[Section~3]{BalmerFavi11}.
\end{remark}

This presentation has the following immediate consequences:

\begin{corollary}\label{493f4da890}
  For a symmetric monoidal \(\infty\)-category~\(\cat{C}\),
  the canonical functor \(\cat{C}\to\h\cat{C}\)
  to its homotopy \(1\)-category
  induces an equivalence on \(\Idem\).
\end{corollary}

\begin{corollary}\label{e16029d69e}
  For a symmetric monoidal \(\infty\)-category~\(\cat{C}\),
  the \(\infty\)-category \(\Idem(\cat{C})\) is
  equivalent to a (unique) poset.
\end{corollary}

By abuse of terminology,
given a symmetric monoidal \(\infty\)-category,
we simply regard \(\Idem(\cat{C})\) as a poset.

\subsection{Coidempotent objects}\label{ss:c}

First,
recall that
any symmetric monoidal structure
determines a symmetric monoidal structure on its opposite \(\infty\)-category;
see~\cite[Remark~2.4.2.7]{LurieHA}.
We define
\begin{align*}
  \cIdem(\cat{C})&=
  \Idem(\cat{C}^{\op})^{\op},&
  \cCAlg_{\idem}(\cat{C})&=
  \CAlg_{\idem}(\cat{C}^{\op})^{\op},&
  \cLoc_{\sm}(\cat{C})&=
  \Loc_{\sm}(\cat{C}^{\op})^{\op}\rlap.
\end{align*}
As we have seen in \cref{ss:e},
they are equivalent.
Objects in those \(\infty\)-categories
are concretely described as follows:

\begin{definition}\label{859fe7c317}
  Let \(\cat{C}\) be a symmetric monoidal \(\infty\)-category.
  A \emph{coidempotent object}
  is an object \(c\colon C\to\unit\)
  of the overcategory \(\cat{C}_{/\unit}\)
  such that
  \(C\otimes c\colon C\otimes C\to C\otimes\unit\simeq C\)
  is an equivalence.
  By abuse of notation,
  we often simply write~\(C\).
  A cocommutative coalgebra object~\(C\)
  is called idempotent\footnote{Here
    “coidempotent coalgebra” can also be a reasonable choice
    of terminology,
    but note that for a bialgebra object,
    its underlying algebra is idempotent
    if and only if
    so is its underlying coalgebra.
  } if its comultiplication morphism \(C\to C\otimes C\) is an equivalence.
  We call a colocalization \(T\colon\cat{C}\to\cat{C}\) \emph{smashing}
  if \(T\) is equivalent to \(T(\unit)\otimes\X\).
\end{definition}

\begin{remark}\label{5d9376b79e}
  A coidempotent object is called 
  a “left idempotent” in~\cite{BalmerFavi11}.
\end{remark}

\begin{remark}\label{cart_uf}
  If
  the unit~\(\unit\) in
  a symmetric monoidal \(\infty\)-category~\(\cat{C}\)
  is final,
  there is a canonical fully faithful embedding
  \(\cIdem(\cat{C})\hookrightarrow\cat{C}\).
\end{remark}

\begin{example}\label{cart_sm}
  For any \(\infty\)-category~\(\cat{C}\)
  having finite products,
  we can consider
  the cartesian symmetric monoidal structure
  on it;
  see~\cite[Section~2.4.1]{LurieHA}.
  By \cref{cart_uf},
  \(\cIdem(\cat{C})\) can be regarded as a full subcategory.
  Moreover in this case
  it is identified with
  the poset of subterminal objects,
  i.e., objects~\(C\) such
  that the morphisms \(C\to{*}\)
  to the final object are monomorphic.
\end{example}

\begin{example}\label{topos}
  Consider 
  the cartesian symmetric monoidal structure
  on an \(\infty\)-topos~\(\cat{X}\).
  Then by \cref{cart_sm},
  the poset \(\cIdem(\cat{X})\)
  is the frame of open subtoposes of~\(\cat{X}\).
  Let \(X\) denote the corresponding locale.
  The inverse image functor
  of the localic reflection
  \(\Shv(X)\to\cat{X}\)
  is identified with the unit of
  the sheaves–spectrum adjunction
  (\cref{main}).
\end{example}

\subsection{The stable case}\label{ss:ce}

In the following,
recall that
a \emph{stably symmetric monoidal \(\infty\)-category}
is a stable \(\infty\)-category
equipped
with a symmetric monoidal structure
whose tensor product is exact in each variable.

\begin{proposition}\label{bb24a428eb}
  For a stably symmetric monoidal \(\infty\)-category~\(\cat{C}\),
  there is
  a canonical equivalence \(\Idem(\cat{C})\simeq\cIdem(\cat{C})\).
\end{proposition}

\begin{proof}
  Let \(\Idem'(\cat{C})\) (\(\cIdem'(\cat{C})\), respectively)
  be the full subcategory of
  \(\Fun(\Delta^1\times\Delta^1,\cat{C})\)
  spanned by
  cartesian diagrams of the form
  \begin{equation}
    \label{e:51521ac8}
    \begin{tikzcd}
      C\ar[r,"c"]\ar[d]&
      \unit\ar[d,"e"]\\
      0\ar[r]&
      E\rlap,
    \end{tikzcd}
  \end{equation}
  where \(e\) is an idempotent object
  (\(c\) is a coidempotent object, respectively).
  As the obvious forgetful functors
  \(\Idem'(\cat{C})\to\Idem(\cat{C})\)
  and \(\cIdem'(\cat{C})\to\cIdem(\cat{C})\)
  are equivalences,
  we need to show that these two \(\infty\)-categories coincide.
  By symmetry, it suffices to show that
  \(\Idem'(\cat{C})\subset\cIdem'(\cat{C})\) holds.
  We consider an arbitrary cartesian diagram
  \cref{e:51521ac8} such that
  \(e\) is an idempotent object.
  We wish to show that
  \(C\otimes c\colon C\otimes C\to C\otimes\unit\) is an equivalence.
  To show this, we need to show that its cofiber
  \(C\otimes E\) is zero.
  This follows from the observation that
  \(C\otimes E\) is also a fiber of
  \(e\otimes E\colon \unit\otimes E\to E\otimes E\),
  which is an equivalence,
  since \(e\) is an idempotent object.
\end{proof}

In the stable situation, we have yet another way to
describe a smashing localization.

\begin{definition}\label{e1bda94d50}
  Let \(\cat{C}\) be a stably symmetric monoidal \(\infty\)-category.
  Its full subcategory is called a \emph{smashing ideal}
  if it is the kernel of some smashing localization.
  Let \(\Idl_{\sm}(\cat{C})\) denote the (essential) poset
  of smashing ideals.
\end{definition}

\begin{proposition}\label{0d0edeb2ed}
  For a stably symmetric monoidal \(\infty\)-category~\(\cat{C}\),
  the kernel functor
  gives an equivalence
  from
  \(\Loc_{\sm}(\cat{C})\)
  to \(\Idl_{\sm}(\cat{C})\).
\end{proposition}

\begin{proof}
  Let \(\Sub(\cat{C})\) denote the poset
  of replete\footnote{We say that a full subcategory is called \emph{replete}
    if it is closed under equivalences.
  } full subcategories of~\(\cat{C}\).
  Then the kernel functor
  \(\Loc(\cat{C})\to\Sub(\cat{C})\)
  is fully faithful,
  since in a stable \(\infty\)-category
  a morphism is an equivalence if and only if
  its fiber is zero.
  The desired equivalence is obtained
  as a restriction of this inclusion.
\end{proof}

\section{The unstable smashing spectrum}\label{s:usm}

In this section,
we define the unstable smashing spectrum functor~\(\Sm\) and
prove the main theorem (\cref{main}),
which says that \(\Sm\) is right adjoint to
the sheaves functor \(\Shv\).
We first see 
in \cref{ss:lat} that
the poset of idempotent coalgebras
of a presentably symmetric monoidal \(\infty\)-category
is a frame up to size issues.
Then in \cref{ss:small},
we show that it is small,
which enables us to define
the functor~\(\Sm\).
In \cref{ss:adj},
we prove \cref{main}.
In the proof,
we need a technical statement
about sheaves on a locale,
which we show in \cref{ss:shv_loc}.

\subsection{Lattice-theoretic properties of \texorpdfstring{\(\cIdem\)}{cIdem}}\label{ss:lat}

In this section,
we show
how properties
on a symmetric monoidal \(\infty\)-category~\(\cat{C}\)
correspond to
lattice-theoretic properties on \(\cIdem(\cat{C})\).
First, we recall some standard notions from lattice theory:

\begin{definition}\label{7014181f28}
  A (bounded) \emph{lower semilattice} is a poset having finite meets.
  A morphism between lower semilattices
  is a function preserving finite meets,
  which is automatically order preserving.
  We write \(\Cat{SLat}\)
  for the category of lower semilattices.
  From now on, we simply call them \emph{semilattices}.
\end{definition}

\begin{remark}\label{9d77cd31d7}
  Lattice theorists typically favor upper semilattices
  over lower semilattices.
\end{remark}

\begin{definition}\label{618aee4b9e}
  A (bounded) \emph{distributive lattice}
  is a semilattice
  that has finite joins and satisfies
  the distributivity law
  \begin{equation}
    \label{e:dist}
    \bigvee_{x'\in S}(x\wedge x')
    =
    x\wedge\bigvee_{x'\in S}x'
  \end{equation}
  for any element \(x\in P\) and any finite set \(S\subset P\).
  A morphism between distributive lattices is a semilattice morphism
  that preserves finite joins.
  By replacing “finite” with “directed” and “arbitrary”
  in this definition,
  we obtain the notions of a \emph{preframe} and a \emph{frame},
  respectively.
  We write
  \(\Cat{DLat}\), \(\Cat{PFrm}\), and \(\Cat{Frm}\)
  for the category of
  distributive lattices,
  preframes, and
  frames, respectively.
  We have an equality \(\Cat{Frm}=\Cat{DLat}\cap\Cat{PFrm}\)
  of subcategories of the category of posets \(\Cat{Poset}\).
\end{definition}

\begin{example}\label{d24c046ed8}
  For any topological space,
  the poset of open subsets is a frame.
  Moreover, any continuous map
  between topological spaces 
  determines a frame morphism
  in the opposite direction.
  The category of sober topological spaces
  is fully faithfully embedded into
  the category of \emph{locales} \(\Cat{Loc}\),
  which is defined as the opposite of~\(\Cat{Frm}\).
  We refer the reader to~\cite[Part~II]{Johnstone82}
  for an introductory account of locale theory.
\end{example}

We then recall the following notation from \(\infty\)-category theory:

\begin{definition}\label{99dd5910bf}
  Let \(\cls{K}\) be a large collection of small \(\infty\)-categories.
  We write \(\widehat{\Cat{Cat}}(\cls{K})\)
  for the \(\infty\)-category
  of large \(\infty\)-categories having \(\cat{I}\)-indexed colimits for any \(\cat{I}\in\cls{K}\)
  and functors preserving \(\cat{I}\)-indexed colimits for any \(\cat{I}\in\cls{K}\);
  see~\cite[Section~5.3.6]{LurieHTT}.
  Moreover,
  there is a standard symmetric monoidal structure on it
  such that
  a morphism \(\cat{C}_1\otimes\dotsb\otimes\cat{C}_n\to\cat{C}'\)
  corresponds to
  a functor \(\cat{C}_1\times\dotsb\times\cat{C}_n\to\cat{C}'\)
  that preserves \(\cat{I}\)-indexed colimits in each variable for each \(\cat{I}\in\cls{K}\);
  see~\cite[Section~4.8.1]{LurieHA}.
\end{definition}

\begin{example}\label{f0f8a8e97f}
  Concretely,
  an object of
  \(\CAlg(\widehat{\Cat{Cat}}(\cls{K}))\) is
  a symmetric monoidal \(\infty\)-category
  such that for each \(\cat{I}\in\cls{K}\),
  its underlying \(\infty\)-category has \(\cat{I}\)-indexed colimits
  and
  its tensor operations preserve \(\cat{I}\)-indexed colimits in each variable.
  A morphism in \(\CAlg(\widehat{\Cat{Cat}}(\cls{K}))\)
  is a symmetric monoidal functor
  preserving \(\cat{I}\)-indexed colimits for any \(\cat{I}\in\cls{K}\).
\end{example}

We come to the main result of this section:

\begin{theorem}\label{49a2f304a6}
  We consider the composite
  \begin{equation}
    \label{e:45e1ca27}
    \phantom,
    \CAlg\bigl(\widehat{\Cat{Cat}}(\cls{K})\bigr)
    \to
    \CAlg\bigl(\widehat{\Cat{Cat}}\bigr)
    \xrightarrow{\cIdem}
    \widehat{\Cat{Poset}}
    ,
  \end{equation}
  where the first arrow is induced by
  the (lax symmetric monoidal) forgetful functor
  \(\widehat{\Cat{Cat}}(\cls{K})\to\widehat{\Cat{Cat}}\).
  \begin{enumerate}
    \item\label{i:slat}
      In general, \cref{e:45e1ca27} lands in \(\widehat{\Cat{SLat}}\).
    \item\label{i:dlat}
      If \(\cls{K}\) contains all finite \(\infty\)-categories,
      \cref{e:45e1ca27} lands in \(\widehat{\Cat{DLat}}\).
    \item\label{i:pfrm}
      If \(\cls{K}\) contains all filtered \(\infty\)-categories,
      \cref{e:45e1ca27} lands in \(\widehat{\Cat{PFrm}}\).
    \item\label{i:frm}
      If \(\cls{K}\) contains all small \(\infty\)-categories,
      \cref{e:45e1ca27} lands in \(\widehat{\Cat{Frm}}\).
  \end{enumerate}
\end{theorem}

\begin{proof}
  We first prove~\cref{i:slat}.
  Let \(\cat{C}\) be a symmetric monoidal \(\infty\)-category.
  First, \({\id}\colon\unit\to\unit\) is
  the largest element of \(\cIdem(\cat{C})\),
  since it has a morphism from any coidempotent object.
  Let \(c\colon C\to\unit\) and \(c'\colon C'\to\unit\) be two coidempotent objects.
  Then \(c\otimes c'
  \colon C\otimes C'\to\unit\otimes\unit\simeq\unit\)
  is a coidempotent object as
  \(
    (C\otimes C')\otimes(c\otimes c')
    \simeq
    (C\otimes c)\otimes(C'\otimes c)
  \)
  is an equivalence.
  By construction,
  any symmetric monoidal functor
  determines a semilattice morphism.

  We then prove~\cref{i:dlat}.
  Let \(\cat{C}\) be an object of \(\CAlg(\widehat{\Cat{Cat}}(\cat{K}))\).
  First, let \(\emptyset\) denote the initial object.
  By assumption,
  \(\emptyset\otimes C\) is also initial for any object \(C\in\cat{C}\).
  Hence
  the morphism from the initial object \(\emptyset\to\unit\)
  is a coidempotent object
  and
  is furthermore the least element of \(\cIdem(\cat{C})\)
  by the description of finite meets in \(\cIdem(\cat{C})\) we have obtained above.
  We also have the distributive law~\cref{e:dist}
  for \(S=\emptyset\).
Next, we construct binary joins.
  Let \(c\colon C\to\unit\) and \(c'\colon C'\to\unit\) be two coidempotent objects.
  We consider the diagram
  \begin{equation}
    \label{e:join}
    \begin{tikzcd}
      C\otimes C'\ar[r,"C\otimes c'"]\ar[d,"c\otimes C'"']&
      C\ar[d,"c"]\\
      C'\ar[r,"c'"]&
      \unit\rlap,
    \end{tikzcd}
  \end{equation}
  where we implicitly identify
  \(C\otimes\unit\) and \(\unit\otimes C'\)
  with~\(C\) and~\(C'\), respectively.
  We let \(c\vee c'\colon C\vee C'\to\unit\)
  denote the morphism obtained
  from this diagram
  by taking a pushout.
  We wish to show that \(c\vee c'\) is a coidempotent object;
  i.e.,
  the morphism \((C\vee C')\otimes(c\vee c')\)
  is an equivalence.
  By the definition of \(C\vee C'\)
  and the assumption on~\(\cat{C}\),
  this fits into a cartesian diagram
  \begin{equation*}
    \begin{tikzcd}
      (C\otimes C')\otimes(c\vee c')\ar[r]\ar[d]&
      C\otimes(c\vee c')\ar[d]\\
      C'\otimes(c\vee c')\ar[r]&
      (C\vee C')\otimes(c\vee c')
    \end{tikzcd}
  \end{equation*}
  in \(\Fun(\Delta^1,\cat{C})\).
  Hence it suffices to show that
  the other three vertices are equivalences.
  By symmetry
  and
  \(
  (C\otimes C')\otimes(c\vee c')
  \simeq
  C'\otimes(C\otimes(c\vee c'))
  \),
  it suffices to show that
  \(C\otimes(c\vee c')
  \colon
  C\otimes(C\vee C')
  \to
  C\otimes\unit\simeq C\) is an equivalence.
  Equivalently, it suffices to show that
  \(C\otimes\text{\cref{e:join}}\) is a pushout square,
  but it is clear as \(c\) is a coidempotent object.
  Therefore,
  we have seen that \(c\vee c'\)
  is a coidempotent object.
  Also,
  \(c\vee c'\) is a join of~\(c\) and~\(c'\)
  by construction.
We then wish to check the distributive law \cref{e:dist}
  when \(S\) consists of two elements.
  Let
  \(c''\colon C''\to\unit\) be a coidempotent object
  and we wish to show that
  the canonical morphism
  \begin{equation*}
    (c\otimes c'')
    \vee
    (c'\otimes c'')
    \to
    (c\vee c')\otimes c''
  \end{equation*}
  in \(\cIdem(\cat{C})\)
  is an equivalence,
  but this follows from
  the observation that
  \begin{equation}
    \label{e:0e6ad386}
    \begin{tikzcd}
      C\otimes C'\ar[r,"C\otimes c'"]\ar[d,"c\otimes C'"']&
      C\ar[d,"c"]\\
      C'\ar[r,"c'"]&
      C\vee C'
    \end{tikzcd}
  \end{equation}
  remains a pushout square after
  applying \(\X\otimes C''\),
  which can be seen from the assumption on~\(\cat{C}\).
  The functoriality follows from
  the observation that
  this construction of finite joins only uses tensor operations and finite colimits,
  which are preserved by morphisms in \(\CAlg(\widehat{\Cat{Cat}}(\cls{K}))\).

  We then check \cref{i:pfrm}.
  Consider a directed family \(c_I\colon C_I\to\unit\)
  indexed by~\(\cat{I}\),
  which is in~\(\cls{K}\) by assumption.
  We write \(c\colon C\to\unit\) for its colimits
  taken in \(\Fun(\Delta^1,\cat{C})\),
  where we use the directedness of~\(\cat{I}\)
  to identify its target with~\(\unit\).
  Since
  the diagonal \(\cat{I}\to\cat{I}\times\cat{I}\) is cofinal,
  we see that
  \begin{equation*}
    C\otimes c
    \simeq
    \biggl(\injlim_{I\in\cat{I}}C_I\biggr)
    \otimes
    \biggl(\injlim_{I'\in\cat{I}}c_{I'}\biggr)
    \simeq
    \injlim_{(I,I')\in\cat{I}\times\cat{I}}
    (C_I\otimes c_{I'})
    \simeq
    \injlim_{I\in\cat{I}}
    (C_I\otimes c_I)
  \end{equation*}
  is an equivalence.
  Hence \(c\) is a coidempotent object.
  The distributivity for directed subsets
  and
  the functoriality
  are clear by assumption.

  Finally, \cref{i:frm} follows from~\cref{i:dlat} and~\cref{i:pfrm}.
\end{proof}

\subsection{Smallness}\label{ss:small}

By~\cref{i:frm} of \cref{49a2f304a6},
we have a functor
\begin{equation*}
  {\cIdem}\colon
  \CAlg(\Cat{Pr})\to\widehat{\Cat{Frm}}.
\end{equation*}
In this subsection,
we see that this actually factors through \(\Cat{Frm}\),
the category of small frames.

\begin{remark}\label{e9e6e9cbb3}
  Much work has been done on the smallness
  of a certain class of localizations in an \(\infty\)-category
  with various accessibility conditions;
  see, e.g.,~\cite{CasacubertaGutierrezRosicky14} and references therein.
  However, these results do not apply to our current situation
  (except for the stable case)
  as we are looking at colocalizations instead.
\end{remark}

\begin{theorem}\label{small}
  Let \(\cat{C}\) be a symmetric monoidal \(\infty\)-category.
  If
  \(\cat{C}\) is accessible
  and
  \({\otimes}\colon\cat{C}\times\cat{C}\to\cat{C}\) is accessible,
  then \(\cIdem(\cat{C})\) is essentially small.
\end{theorem}

\begin{remark}\label{d322807483}
  In general,
  since filtered \(\infty\)-categories are sifted,
  a functor of the form \(\cat{C}\times\cat{D}\to\cat{E}\)
  that preserves filtered colimits in each variable
  preserves filtered colimits.
  Hence
  the condition of \cref{small} is satisfied
  for an object of \(\CAlg(\Cat{Pr})\).
\end{remark}

\Cref{small} is deduced from the following two lemmas:

\begin{lemma}\label{e9559e4a68}
  In the situation of \cref{small},
  \(\cIdem(\cat{C})\) is accessible.
\end{lemma}

\begin{proof}
We first observe that
  \(\cIdem(\cat{C})\) is a pullback
  \begin{equation*}
    \begin{tikzcd}
      \cIdem(\cat{C})\ar[r]\ar[d]&
      \cat{C}_{/\unit}\ar[d,"(C\to\unit)\mapsto(C\to C\otimes C)"]\\
      \cat{C}\ar[r,"\textnormal{diagonal}"]&
      \Fun(\Delta^1,\cat{C})
    \end{tikzcd}
  \end{equation*}
  in the \(\infty\)-category of large \(\infty\)-categories.
  Therefore, by~\cite[Proposition~5.4.6.6]{LurieHTT},
  we need to see
  that the cospan belongs to the \(\infty\)-category
  of accessible \(\infty\)-categories and accessible functors.
  First,
  \(\cat{C}_{/\unit}\)
  and
  \(\Fun(\Delta^1,\cat{C})\)
  are accessible according to~\cite[Corollary~5.4.6.7 and Proposition~5.4.4.3]{LurieHTT},
  respectively.
  (This can also be checked directly.)
  We then choose an infinite regular cardinal~\(\kappa\)
  such that
  \(\cat{C}\) has \(\kappa\)-filtered colimits
  and
  \({\otimes}\colon\cat{C}\times\cat{C}\to\cat{C}\)
  preserves \(\kappa\)-filtered colimits.
  Then both right and bottom functors
  preserve \(\kappa\)-filtered colimits
  and are thus accessible.
\end{proof}

\begin{lemma}\label{94a0eb29dd}
  A large poset that is accessible (when regarded as an \(\infty\)-category)
  is essentially small.
\end{lemma}

\begin{proof}
  In this proof, by “poset” we mean large poset.
  By definition,
  an accessible poset
  is isomorphic to the poset of the form \(\Ind_{\kappa}(P)\)
  for
  an infinite regular cardinal~\(\kappa\)
  and
  a small poset~\(P\).
  As \(\Ind_{\kappa}(P)\) can be identified
  with a full subposet of the poset of full subposets of~\(P\),
  it is essentially small.
\end{proof}

Given \cref{i:frm} of \cref{49a2f304a6},
\cref{small}, and \cref{d322807483},
now
we can introduce the main notion of this paper:

\begin{definition}\label{56fda69d7e}
  The \emph{unstable smashing spectrum} functor is the composite
  \begin{equation*}
    {\Sm}\colon
    \CAlg(\Cat{Pr})\xrightarrow{\cIdem}\Cat{Frm}\xrightarrow{=}\Cat{Loc}^{\op}.
  \end{equation*}
\end{definition}

\begin{remark}\label{f83636aa2e}
  Note that the word “smashing” in the name is still justified
  as its opens correspond to smashing colocalizations.
\end{remark}

\subsection{Sheaves on a locale}\label{ss:shv_loc}

In this section,
we prove the following fact,
which is an ingredient of the proof of the main theorem:

\begin{theorem}\label{ee1f8d640f}
  Let \(X\) be a locale.
  We write \(F\) for its underlying frame.
  For an \(\infty\)-category having limits~\(\cat{C}\),
  a presheaf \(\shf{F}\colon F^{\op}\to\cat{C}\) is a 
  sheaf if and only if
  it satisfies
  the following conditions:
  \begin{enumerate}
    \item\label{i:red}
      The value \(\shf{F}(\bot)\) is the final object.
    \item\label{i:exc}
      For any opens~\(V\) and~\(V'\),
      the square
      \begin{equation*}
        \begin{tikzcd}
          \shf{F}(V\vee V')\ar[r]\ar[d]&
          \shf{F}(V)\ar[d]\\
          \shf{F}(V')\ar[r]&
          \shf{F}(V\wedge V')
        \end{tikzcd}
      \end{equation*}
      is cartesian.
    \item\label{i:fil}
      The canonical morphism
      \begin{equation*}
        \shf{F}\Bigl(\bigvee D\Bigr)
        \to
        \projlim_{U\in D}
        \shf{F}(U)
      \end{equation*}
      is an equivalence
      for any directed subset~\(D\subset F\).
  \end{enumerate}
\end{theorem}

We need several small lemmas to prove this theorem:

\begin{lemma}\label{85ede2d6c6}
  In the situation of \cref{ee1f8d640f},
  a presheaf \(\shf{F}\) is a sheaf
  if and only if
  the canonical morphism
  \begin{equation*}
    \label{e:21165d07}
    \shf{F}\Bigl(\bigvee R\Bigr)
    \to
    \projlim_{U\in R}
    \shf{F}(U)
  \end{equation*}
  is an equivalence
  for any downward-closed subset \(R\subset F\).
\end{lemma}

\begin{proof}
  This follows from the observation
  that
  for any \(V\in F\),
  the overcategory
  \(F_{/V}\) is just the full subposet \(F_{\leq V}\subset F\)
  and hence
  a covering sieve on~\(V\)
  is just a downward-closed subset of~\(F\)
  whose join is~\(V\).
\end{proof}

\begin{lemma}\label{ee6e60039b}
  Let \(P\) be a poset
  having directed joins
  and \(\cat{C}\) an \(\infty\)-category
  having filtered colimits.
  Then
  a functor
  \(F\colon P\to\cat{C}\)
  preserves filtered colimits
  if and only if
  the canonical morphism
  \begin{equation*}
    \injlim_{p\in D}F(p)
    \to
    F\Bigl(\bigvee D\Bigr)
  \end{equation*}
  is an equivalence
  for any directed subset \(D\subset P\).
\end{lemma}

\begin{proof}
  Since the “only if” direction is obvious,
  we prove the “if” direction.
  Consider an arbitrary
  filtered diagram
  \(\cat{I}\to P\).
  We wish to show that
  \(F\) preserves the colimit of this diagram.
  By~\cite[Proposition~5.3.1.18]{LurieHTT},
  we can assume that \(\cat{I}\) is just a directed poset.
  Let \(D\) be its essential image.
  Since \(\cat{I}\) is directed,
  so is \(D\).
  Then we claim that \(\cat{I}\to D\) is cofinal.
  To prove this,
  we need to show that
  \(\cat{I}\times_{D}D_{p/}\) is weakly contractible
  for any \(p\in D\),
  but it is directed.
  Therefore,
  \(F\) preserves the colimit of this diagram
  by assumption.
\end{proof}

In the following lemma,
recall that
\(\Lambda_0^2\) denotes (the nerve of) the poset
\(\{1\gets0\to2\}\):

\begin{lemma}\label{exc_two}
  Let \(P\) be a poset
  and \(p\) and \(p'\in P\)
  (not necessarily distinct) elements having meet \(p\wedge p'\).
  Suppose that
  \(P=P_{\leq p}\cup P_{\leq p'}\) holds.
  Then the functor
  \(F\colon\Lambda_0^2\to P\)
  determined by
  \begin{align*}
    0&\mapsto p\wedge p',&
    1&\mapsto p,&
    2&\mapsto p'
  \end{align*}
  is cofinal.
\end{lemma}

\begin{proof}
  We need to see that for any \(q\in P\),
  the full subposet
  \(I=\{n\in\Lambda_0^2\mid F(n)\geq q\}\)
  is weakly contractible.
  As \(P=P_{\leq p}\cup P_{\leq p'}\),
  we assume \(q\leq p\).
  When \(q\leq p'\),
  we have \(q\leq p\wedge p'\) and
  thus \(I\) is \(\Lambda_0^2\).
  When \(q\nleq p'\),
  we have \(q\nleq p\wedge p'\) and 
  thus \(I\) is \(\{1\}\simeq{*}\).
  Both are weakly contractible.
\end{proof}

In the proof,
we use the following notation
used in~\cite[Section~5.1]{verdier-asc}:

\begin{definition}\label{e5d27e0fc4}
  For a poset~\(P\),
  let \(\Down(P)\) denote
  the poset of downward-closed subsets (aka sieves).
  We write \(\Down^{\fin}(P)\) for
  the full subposet spanned by
  finitely generated downward-closed subsets;
  i.e., \(R\in\Down(P)\) belongs to \(\Down^{\fin}(P)\)
  if and only if there is a finite subset \(S\subset P\)
  such that \(R\) is the downward closure of~\(S\).
\end{definition}

\begin{proof}[Proof of \cref{ee1f8d640f}]
  First suppose that \(\shf{F}\) is a sheaf.
  Then \cref{i:red}
  and \cref{i:fil} follow from \cref{85ede2d6c6}
  by considering the cases
  \(R=\emptyset\) and~\(R=D\), respectively.
  We prove \cref{i:exc} by taking
  \(R\) to be the downward closure of \(\{V,V'\}\)
  in \cref{85ede2d6c6}.
  Since
  the functor \(\Lambda_0^2\to R\)
  determined by \(0\mapsto V\wedge V'\),
  \(1\mapsto V\), and~\(2\mapsto V'\)
  is cofinal
  by \cref{exc_two},
  the desired claim follows.

  Then we prove the converse.
  Let \(\shf{F}\colon F^{\op}\to\cat{C}\)
  be a functor satisfying the conditions
  listed in the statement.
  First,
  note that
  \cref{i:fil} implies that
  \(\shf{F}\) preserves cofiltered limits
  by \cref{ee6e60039b}.
  Let \(R\subset F\) be an arbitrary 
  downward-closed subset.
  We consider the diagram
  \begin{equation*}
    \begin{tikzcd}
      \shf{F}\bigl(\bigvee R\bigr)\ar[r]\ar[d]&
      \projlim_{U\in R}
      \shf{F}(U)\ar[d]\\
      \projlim_{R'\in\Down^{\fin}(R)}
      \shf{F}\bigl(\bigvee R'\bigr)\ar[r]&
      \projlim_{R'\in\Down^{\fin}(R)}
      \projlim_{U\in R'}
      \shf{F}(U)
    \end{tikzcd}
  \end{equation*}
  in~\(\cat{C}\).
  By \cref{85ede2d6c6},
  it suffices to show that the top arrow is an equivalence
  to conclude the proof.
  The vertical arrows are equivalences,
  since \(\shf{F}\) preserves cofiltered limits.
  Hence it remains to show that
  the bottom arrow is an equivalence.

Hence
  we are reduced to showing that for any finite subset~\(S\subset F\),
  the canonical morphism
  \(\shf{F}\bigl(\bigvee S\bigr)
  =\shf{F}\bigl(\bigvee R\bigr)
  \to\projlim_{U\in R}\shf{F}(U)\) is an equivalence,
  where \(R\) is the downward closure of~\(S\).
  We prove this claim by induction on~\(\#S\).
  When \(\#S=0\), the claim is \cref{i:red}.
  We assume \(\#S\geq1\).
  We pick an element \(V\in S\)
  and consider
  \(S'=S\setminus\{V\}\)
  and
  \(S''=\{V\wedge U\mid U\in S'\}\).
  We write \(R\), \(R'\), and~\(R''\)
  for the downward closures of \(S\), \(S'\), and~\(S''\),
  respectively.
  Then we consider the diagram
  \begin{equation*}
    \begin{tikzcd}[row sep=tiny,column sep=tiny]
      \shf{F}\bigl(\bigvee S\bigr)\ar[rr]\ar[dr]\ar[dd]&
      {}&
      \projlim_{U\in R}\shf{F}(U)\ar[dd]\ar[dr]&
      {}\\
      {}&
      \shf{F}(V)\ar[rr,crossing over]&
      {}&
      \projlim_{U\leq V}\shf{F}(U)\ar[dd]\\
      \shf{F}\bigl(\bigvee S'\bigr)\ar[rr]\ar[dr]&
      {}&
      \projlim_{U\in R'}\shf{F}(U)\ar[dr]&
      {}\\
      {}&
      \shf{F}\bigl(\bigvee S''\bigr)\ar[rr]\ar[uu,<-,crossing over]&
      {}&
      \projlim_{U\in R''}\shf{F}(U)
    \end{tikzcd}
  \end{equation*}
  in~\(\cat{C}\).
  By inductive hypothesis and
  \(\#S''\leq\#S'<\#S\),
  the horizontal arrows are equivalences except the top one.
  By distributivity,
  we have
  \begin{equation*}
    V\wedge\Bigl(\bigvee S'\Bigr)
    =
    V\wedge\biggl(\bigvee_{U\in S'}U\biggr)
    =
    \bigvee_{U\in S}(V\wedge U)
    =
    \bigvee S''
  \end{equation*}
  and hence the left square is cartesian by \cref{i:exc}.
  By applying~\cite[Proposition~4.4.2.2]{LurieHTT} to
  the decomposition
  \(R=
  \{U\in F\mid U\leq V\}
  \amalg_{R''}R'\),
  we know that
  the right square is also cartesian.
  Therefore, the top horizontal arrow
  is written as a limit of the other horizontal arrows
  and hence is an equivalence.
\end{proof}

\begin{remark}\label{f31b7b3f6f}
  The final part of the proof shows that
  a presheaf on a distributive lattice is a sheaf
  for the finite cover topology
  if and only if
  it satisfies \cref{i:red} and~\cref{i:exc}.
  Via Stone duality between
  distributive lattices and
  coherent topological spaces (aka spectral spaces),
  this variant is related to~\cite[Theorem~7.3.5.2]{LurieHTT}.
\end{remark}

\subsection{The sheaves--spectrum adjunction}\label{ss:adj}

We finally construct
the sheaves–spectrum adjunction (\cref{main}) in this section.
We first give its toy version,
which gives an external characterization of \(\cIdem\).
To state it,
we note
the following:

\begin{lemma}\label{08a2ef6f35}
  By regarding a semilattice as
  a cartesian symmetric monoidal \(\infty\)-category,
  we have a fully faithful embedding
  \(\iota\colon\Cat{SLat}\hookrightarrow\CAlg(\Cat{Cat})\).
\end{lemma}

\begin{proof}
  This is an immediate consequence
  of~\cite[Corollary~2.4.1.9]{LurieHA}.
\end{proof}

\begin{proposition}\label{adj}
  The functor~\(\iota\) in \cref{08a2ef6f35} is
  left adjoint to
  \(\cIdem\)
  given by (the obvious variant of)~\cref{i:slat}
  of \cref{49a2f304a6}:
  \begin{equation*}
    \begin{tikzcd}
      \Cat{SLat}\ar[r,shift left,hook,"\iota"]&
      \CAlg(\Cat{Cat})\rlap.\ar[l,shift left,"\cIdem"]
    \end{tikzcd}
  \end{equation*}
\end{proposition}

We need some lemmas:

\begin{lemma}\label{2ed3114c1a}
  Let \(\CAlg(\Cat{Cat})'\) be
  the full subcategory of \(\CAlg(\Cat{Cat})\)
  spanned by symmetric monoidal \(\infty\)-categories
  whose units are final.
  Then this is a coreflective subcategory;
  i.e., the inclusion has a right adjoint.
  The coreflector is objectwise given by
  \(\cat{C}\mapsto\cat{C}_{/\unit}\).
\end{lemma}

\begin{proof}
  Let \(\Cat{Cat}_*\) be the \(\infty\)-category
  of pointed \(\infty\)-categories
  and \(\Cat{Cat}_*'\) the full subcategory
  spanned by pointed \(\infty\)-categories
  whose distinguished objects are final.
  By~\cite[Lemma~5.2.4.11]{LurieHA},
  this is a coreflective subcategory
  and
  the coreflector is given by \((\cat{C},C)\mapsto\cat{C}_{/C}\).
  We have functors
  \(
    \Cat{Cat}_*'\to\Cat{Cat}_*\to\Cat{Cat}
  \)
  preserving finite limits.
  We consider the cartesian symmetric monoidal structures on them
  and apply \(\CAlg(\X)\).
  By~\cite[Remark~2.2.1.5]{LurieHA},
  the first functor induces a fully faithful inclusion
  having a right adjoint.
  The second functor induces an equivalence
  as already seen after applying \(\Alg_{\E_0}(\X)\).
  Therefore the composite \(\CAlg(\Cat{Cat}_*')\to\CAlg(\Cat{Cat})\)
  is fully faithful and the image is identified with
  \(\CAlg(\Cat{Cat})'\) and hence the desired claim follows.
\end{proof}

\begin{lemma}\label{572f1e9d7a}
  For a symmetric monoidal \(\infty\)-category~\(\cat{C}\),
  the counit \(\cat{C}_{/\unit}\to\cat{C}\)
  of the adjunction in \cref{2ed3114c1a}
  induces an equivalence on \(\cIdem\).
\end{lemma}

\begin{proof}
  This immediately follows from the definition of \(\cIdem\).
\end{proof}

\begin{proof}[Proof of \cref{adj}]
  We use the \(\infty\)-category \(\CAlg(\Cat{Cat})'\)
  in the statement of \cref{2ed3114c1a}.
  Note that \(\iota\)
  factors through \(\CAlg(\Cat{Cat})'\).
  By \cref{572f1e9d7a},
  we need to show that
  \(\Cat{SLat}\to\CAlg(\Cat{Cat})'\) is left adjoint
  to the restriction of \(\cIdem\)
  to \(\CAlg(\Cat{Cat})'\).

  By \cref{cart_uf}
  and the description of
  finite meets in \(\cIdem\)
  given in the proof of \cref{49a2f304a6},
  for \(\cat{C}\in\CAlg(\Cat{Cat})'\),
  we have a symmetric monoidal inclusion
  \(\iota(\cIdem(\cat{C}))\hookrightarrow\cat{C}\).
  It suffices to show that for any semilattice~\(P\),
  the induced morphism
  \begin{equation*}
    \Map_{\Cat{SLat}}(P,\cIdem(\cat{C}))
    \xrightarrow{\simeq}
    \Map_{\CAlg(\Cat{Cat})}(\iota(P),\iota(\cIdem(\cat{C})))
    \hookrightarrow
    \Map_{\CAlg(\Cat{Cat})}(\iota(P),\cat{C})
  \end{equation*}
  is an equivalence.
  So it suffices to observe that
  any symmetric monoidal functor \(\iota(P)\to\cat{C}\)
  factors through \(\iota(\cIdem(\cat{C}))\),
  which is clear as each element of~\(P\)
  determines a coidempotent.
\end{proof}

We then prove \cref{main}.
More precisely,
we prove the following reformulation:

\begin{theorem}\label{main_2}
  The functor given by \cref{i:frm}
  of \cref{49a2f304a6}
  is right adjoint to~\(\Shv\):
  \begin{equation*}
    \begin{tikzcd}
      \Cat{Frm}\ar[r,shift left,hook,"\Shv"]&
      \CAlg(\Cat{Pr})\rlap.\ar[l,shift left,"\cIdem"]
    \end{tikzcd}
  \end{equation*}
\end{theorem}

We use the following notion in the proof:

\begin{definition}\label{0983741d9b}
  By~\cite[Remark~4.8.1.8 and Proposition~4.8.1.15]{LurieHA},
  the functor \({\PShv}\colon\Cat{Cat}\to\Cat{Pr}\)
  is refined to a symmetric monoidal functor.
  Hence
  for a symmetric monoidal \(\infty\)-category~\(\cat{C}\),
  there is
  a canonical presentable symmetric monoidal structure
  on \(\PShv(\cat{C})\).
  This is called
  the \emph{Day convolution} symmetric monoidal structure.
\end{definition}

\begin{lemma}\label{day}
  Let \(\cat{C}\) be an \(\infty\)-category with finite products.
  We consider the cartesian symmetric monoidal structure
  on~\(\cat{C}\).
  Then the Day convolution symmetric monoidal structure
  on \(\PShv(\cat{C})\) is again cartesian.
\end{lemma}

\begin{proof}
  We instead prove
  that the cartesian symmetric monoidal structure on \(\PShv(\cat{C})\)
  satisfies the characteristic properties of the Day convolution
  given in~\cite[Corollary~4.8.1.12]{LurieHA}.
  We need to prove
  that
  the Yoneda embedding
  \(\cat{C}\to\PShv(\cat{C})\) preserves finite products
  and
  that
  \({\X\times\X}\colon\PShv(\cat{C})\times\PShv(\cat{C})\to\PShv(\cat{C})\)
  preserves colimits in each variable,
  but both are clear.
\end{proof}

\begin{proof}[Proof of \cref{main_2}]
  As explained in \cref{topos},
  we can canonically identify
  \(\cIdem(\Shv(F))\) with~\(F\)
  for a frame~\(F\).
  We wish to show that this gives
  a unit for the desired adjunction.
  Concretely,
  for any frame~\(F\)
  and a presentably symmetric monoidal \(\infty\)-category~\(\cat{C}\),
  we need to show that
  the composite
  \begin{equation*}
    \Map_{\CAlg(\Cat{Pr})}(\Shv(F),\cat{C})
    \to
    \Map_{\Cat{Frm}}(\cIdem(\Shv(F)),\cIdem(\cat{C}))
    \xrightarrow{\simeq}
    \Map_{\Cat{Frm}}(F,\cIdem(\cat{C}))
  \end{equation*}
  is an equivalence.
  We show that the first map is an equivalence.

  We note that
  the cartesian symmetric monoidal structure on~\(\Shv(X)\)
  is a localization of that on~\(\PShv(F)\)
  in the sense of~\cite[Proposition~2.2.1.9]{LurieHA},
  since the sheafification functor preserves finite products.
  The symmetric monoidal functors
  \(\iota(F)\to\PShv(F)\to\Shv(F)\)
  induce the commutative diagram
\begin{equation*}
    \begin{tikzcd}
      \Map_{\CAlg(\Cat{Pr})}(\Shv(F),\cat{C})\ar[r]\ar[d,hook]&
      \Map_{\Cat{Frm}}(\cIdem(\Shv(F)),\cIdem(\cat{C}))\ar[dd,hook']\\
      \Map_{\CAlg(\Cat{Pr})}(\PShv(F),\cat{C})\ar[d,"\simeq"']&
      {}\\
      \Map_{\CAlg(\Cat{Cat})}(\iota(F),\cat{C})\ar[r,"\simeq"]&
      \Map_{\Cat{SLat}}(\cIdem(\iota(F)),\cIdem(\cat{C}))\rlap,
    \end{tikzcd}
  \end{equation*}
  where the hooked arrows are full inclusions
  and
  the bottom left vertical arrow and bottom horizontal arrow
  are equivalences by \cref{day} and \cref{adj}, respectively.
  We need to see that the top arrow,
  which is an inclusion of a full subspace,
  is an equivalence.
  It suffices to show that
  if
  a morphism \(f\colon F\to\cIdem(\cat{C})\)
  of semilattices,
  regarded as a point of the right bottom corner,
  is a frame morphism,
  it comes from the left top corner.
  
  We fix a frame morphism \(f\colon F\to\cIdem(\cat{C})\).
  This gives a symmetric monoidal functor \(\PShv(F)\to\cat{C}\)
  by left Kan extending
  \(F\to\cIdem(\cat{C})\to\cat{C}\).
  By~\cite[Proposition~4.1.7.4]{LurieHA},
  it suffices to show that
  its underlying functor factors through
  the sheafification functor \(\PShv(F)\to\Shv(F)\).
  We regard \(\PShv(F)\to\cat{C}\)
  as a presheaf on~\(F\) valued in~\(\cat{C}^{\op}\).
  With this identification,
  it suffices to show that
  the composite \(F\xrightarrow{f}\cIdem(\cat{C})\to\cat{C}\)
  is a sheaf.
  We prove this by checking the conditions
  listed in \cref{ee1f8d640f}.
  First we see~\cref{i:red}.
  The proof of \cref{49a2f304a6} shows that
  \(\cIdem(\cat{C})\to\cat{C}\) preserves the initial object
  and hence so does
  \(F\to\cat{C}\).
  Similarly,
  \cref{i:exc}
  follows from the construction
  of binary joins in \(\cIdem(\cat{C})\)
  done in the proof of \cref{49a2f304a6};
  cf.~\cref{e:0e6ad386}.
  Finally, \cref{i:fil}
  also follows from the proof of \cref{49a2f304a6},
  where we see that
  \(\cIdem(\cat{C})\to\cat{C}\)
  preserves filtered colimits.
\end{proof}

\begin{remark}\label{dcb0ed51e4}
  Similarly,
  by using the variant of \cref{ee1f8d640f}
  mentioned in \cref{f31b7b3f6f},
  we can prove that the functor
  \(\Cat{DLat}\to\CAlg(\Cat{Rex})\)
  given by
  \(D\mapsto\Shv(\Spec D)\)
  admits a right adjoint,
  which is given by restricting
  \(\cIdem\) (cf.~\cref{i:dlat} of \cref{49a2f304a6}),
  where \(\Cat{Rex}\) denotes
  the \(\infty\)-category
  of
  \(\infty\)-categories having finite colimits
  and
  functors preserving finite colimits.

Note that this version of
  the spectrum
  of a stably symmetric monoidal \(\infty\)-category
  is smaller than the Balmer spectrum;
  e.g.,
  the spectrum of
  the \(\infty\)-category
  of perfect complexes of abelian groups
  is~\(*\).
  In fact, this construction recovers
  the Pierce spectrum of a commutative ring
  (or in general, a connective ring spectrum);
  see~\cite[Example~7.37]{ttg-con}.
\end{remark}

\section{Categorified locales}\label{s:catloc}

\subsection{A categorical presentation}\label{ss:pre}

We recall the notion of (presentable) categorified locales
and prove \cref{pt},
which states that
the \(\infty\)-category of them
can be regarded as a certain full subcategory of
\(\Fun(\Delta^1,\CAlg(\Cat{Pr}))^{\op}\).

\begin{remark}\label{758bfc3136}
  Our categorical presentation theorem
  is less useful in practice
  than
  Lurie-type Tannaka duality reconstruction results for stacks.
  Rather, this theorem just shows
  how categorical the notion of a categorified locale is.
  It also provides us with a justification for the notion.
\end{remark}

As explained in \cref{ss:q},
a categorified locale \((X,\cat{C}_X)\)
consists of a locale~\(X\),
a stable cocomplete closed symmetric monoidal \(\infty\)-category~\(\cat{C}_X\),
and a large locale map \(\Sm(\cat{C}_X)\to X\).
A morphism \((Y,\cat{C}_Y)\to(X,\cat{C}_X)\)
between categorified locales
consists of
a large locale map \(f\colon Y\to X\)
and
a colimit-preserving symmetric monoidal functor \(F\colon\cat{C}_X\to\cat{C}_Y\)
that is compatible with~\(f\) in the sense that
the diagram
\begin{equation*}
  \begin{tikzcd}
    \Sm(\cat{C}_Y)\ar[r,"\Sm(F)"]\ar[d]&
    \Sm(\cat{C}_X)\ar[d]\\
    Y\ar[r,"f"]&
    X
  \end{tikzcd}
\end{equation*}
in \(\widehat{\Cat{Loc}}\)
commutes,
where the vertical maps are the structure morphisms.
We here formally define the \(\infty\)-category
of presentable categorified locales:

\begin{definition}\label{catloc}
We regard \(\Sm\colon\CAlg(\Cat{Pr})\to\Cat{Loc}^{\op}\)
  as a diagram \(\Delta^1\to\widehat{\Cat{Cat}}\).\footnote{The \(\infty\)-category
    \(\Cat{Pr}\),
    which is a~priori very large,
    is essentially large.
  }
  Let \(p\colon\cat{M}\to\Delta^1\) be the corresponding cartesian fibration.
  We define the \(\infty\)-category \(\Cat{CatLoc}\)
  of (unstable) \emph{categorified locales} to be 
  the opposite of the \(\infty\)-category
  \(\Fun_{/\Delta^1}(\Delta^1,\cat{M})\)
  of sections.
  We call it \emph{stable} when \(\cat{C}\)
  is stable
  and
  write \(\Cat{CatLoc}_{\st}\)
  for
  its full subcategory spanned by stable ones.
\end{definition}

\begin{remark}\label{d1bc60fcb8}
  The \(\infty\)-category of
  stable categorified locales
  \(\Cat{CatLoc}_{\st}\) can be identified
  with the undercategory of \(\Cat{CatLoc}\)
  with respect to \(({*},\Cat{Sp})\).
\end{remark}

\begin{remark}\label{presentability}
  The presentability assumption
  in the definition is not a real restriction;
  see \cref{8875b12b73}.
\end{remark}

We prove the following,
which immediately implies \cref{pt}:

\begin{theorem}\label{dfb9195053}
  There is a fully faithful functor
  \(
    \Cat{CatLoc}^{\op}\to\Fun(\Delta^1,\CAlg(\Cat{Pr}))
  \)
  objectwise given by
  associating to a categorified locale \((X,\cat{C})\)
  the morphism \(\Shv(X)\to\cat{C}\)
  corresponding to \(\Sm(\cat{C})\to X\)
  via the adjunction of \cref{main}.
\end{theorem}

\begin{remark}\label{7c98ee819e}
  Let \(k\) be a field.
  In light of \cref{sm_shv_k},
  we can show by
  a similar argument that
  the \(\infty\)-category of categorified locales
  under \(({*},\D(k))\)
  is fully faithfully embedded into
  \(\Fun(\Delta^1,\CAlg_{\D(k)}(\Cat{Pr}))\).
\end{remark}

\begin{proof}
By \cref{main},
  the cartesian fibration~\(\cat{M}\) in 
  \cref{catloc}
  can be regarded as the cocartesian fibration
  corresponding to
  \({\Shv}\in\Fun(\Delta^1,\widehat{\Cat{Cat}})\).
  We regard
  the diagram
  \begin{equation*}
    \begin{tikzcd}
      \Cat{Loc}^{\op}\ar[r,"\Shv",hook]\ar[d,"\Shv"',hook]&
      \CAlg(\Cat{Pr})\ar[d,equal]\\
      \CAlg(\Cat{Pr})\ar[r,equal]&
      \CAlg(\Cat{Pr})
    \end{tikzcd}
  \end{equation*}
  as the morphism
  \({\Shv}\to{\id}_{\CAlg(\Cat{Pr})}\)
  in \(\Fun(\Delta^1,\widehat{\Cat{Cat}})\).
  By unstraightening this,
  we get a fully faithful functor
  \(\cat{M}\to\Delta^1\times\CAlg(\Cat{Pr})\)
  between cocartesian fibrations
  over~\(\Delta^1\)
  preserving cocartesian edges.
  By taking the section \(\infty\)-categories,
  we get the desired functor.
\end{proof}

We record the following corollary:

\begin{corollary}\label{xhurqm}
The \(\infty\)-category~\(\Cat{CatLoc}\)
  has colimits.
  Concretely,
  we have
  \begin{equation*}
    \injlim_{I\in\cat{I}}
    (X_{I},\cat{C}_{I})
    \simeq
    \biggl(X,\projlim_{I\in\cat{I}}\cat{C}_{I}\biggr)
  \end{equation*}
  for a diagram
  \(\cat{I}\to\Cat{CatLoc}\)
  where \(X=\injlim_{I\in\cat{I}}X_{I}\)
  in the category of locales.
\end{corollary}

\begin{proof}
  By \cref{dfb9195053},
  we see that it has colimits.
  It also gives us the formula
  \begin{equation*}
    \injlim_{I\in\cat{I}}
    (X_{I},\cat{C}_{I})
    \simeq
    \biggl(
    \Sm\biggl(\projlim_{I\in\cat{I}}\Shv(X_{I})\biggr),
    \projlim_{I\in\cat{I}}\cat{C}_{I}
    \biggr)
  \end{equation*}
  for a diagram
  \(\cat{I}\to\Cat{CatLoc}\).
  So we have to identify
  the first component of the right-hand side
  with~\(X=\injlim_{I\in\cat{I}}X_{I}\).
  According to~\cite[Proposition~6.3.2.3]{LurieHTT},
  the limit
  \(\projlim_{I\in\cat{I}}\Shv(X_{I})\)
  is the colimit of \(\infty\)-toposes~\(\Shv(X_{I})\).
Hence the desired claim follows from \cref{topos}.
\end{proof}

\bibliographystyle{plain}
\let\SS\oldSS  \newcommand{\yyyy}[1]{}

\end{document}